\documentclass{amsart}
\usepackage{tikz}
\usetikzlibrary{arrows.meta,decorations.pathreplacing}
\usepackage{xcolor}
\usepackage{amssymb,latexsym,amsmath,extarrows}
\usepackage{graphicx,mathrsfs}
\usepackage{hyperref,url}
\numberwithin{equation}{section}

\newtheorem{theorem}{Theorem}[section]
\newtheorem{lemma}[theorem]{Lemma}

\newtheorem{proposition}[theorem]{Proposition}

\newtheorem{definition}[theorem]{Definition}
\newtheorem{corollary}[theorem]{Corollary}

\newtheorem{conjecture}[theorem]{Conjecture}

\newcommand{\de}{\delta}

\newcommand{\De}{\Delta}
\newcommand{\e}{\epsilon}

\newcommand{\vp}{\varphi}

\newcommand{\wh}{\widehat}

\newcommand{\ZR}{\mathbb{R}}
\newcommand{\R}{\mathbb{R}}

\newcommand{\ZQ}{\mathbb{Q}}
\newcommand{\ZZ}{\mathbb{Z}}

\newcommand{\ZN}{\mathbb{N}}
\newcommand{\ZS}{\mathbb{S}}

\newcommand{\Id}{{\bf 1}}

\newcommand{\gap}{\operatorname{gap}}

\newcommand{\cC}{{\mathcal C}}

\newcommand{\supp}{{\rm supp}}

\newcommand{\dist}{{\rm{dist}}}

\begin{document}

\title[Three-fold additive energy bound for points on convex curves]{Near optimal three-fold additive energy bound for points on convex curves}

\author{Adam Cushman} \address{ Adam Cushman\\  Department of Mathematics\\ Indiana University Bloomington, USA} \email{acushma@iu.edu}

\author{Ciprian Demeter} \address{ Ciprian Demeter\\  Department of Mathematics\\ Indiana University Bloomington, USA} \email{demeterc@iu.edu}

\author{Shukun Wu} \address{ Shukun Wu\\  Department of Mathematics\\ Indiana University Bloomington, USA} \email{shukwu@iu.edu}

\begin{abstract}
Let $X\subset\R$ be finite and let $\gamma(t)=(t,f(t))$, where $f$ is strictly convex. 
We show that
\[
    J_3(\gamma(X))
    =\#\{(x_1,\ldots,x_6)\in X^6:\sum_{i=1}^3\gamma(x_i)=\sum_{i=4}^6\gamma(x_i)\}
    \ll_{\epsilon}|X|^{3+\epsilon}.
\]
When specialized to the parabola, our result implies near-optimal estimates for the number of solutions to the diameter-free quadratic Vinogradov system.
As a second application, we settle a conjecture from  \cite{Krishnapur-Kurlberg-Wigman}, \cite{Bombieri-Bourgain} concerning lattice points on dilates of the unit circle. 
As a third application, we
prove that $|A-A|\gg_\e|A|^{5/3-\e}$ and $|A+A|\gg_\e|A|^{8/5-\e}$ for any finite convex sequence $A\subset \R$. 

\end{abstract}

\maketitle

\section{Introduction}

A function $f:\ZR\to\ZR$ is called \emph{strictly convex} if for every pair of
distinct points $x,y\in \ZR$ and every $\lambda\in(0,1)$,
\[
f\bigl(\lambda x+(1-\lambda)y\bigr)<
\lambda f(x)+(1-\lambda)f(y).
\]
Given a strictly convex function $f$, define the strictly convex curve
\[
   \gamma(t):=(t,f(t)), \qquad t\in\ZR.
\]
For a finite $Y\subset \ZR^n$ and an integer $k\ge 1$, we consider the $k$-fold energy
\begin{equation}
\label{J-k}
    J_k(Y)
    :=\#\{(y_1,\ldots,y_{2k})\in Y^{2k}:\sum_{i=1}^{k}y_i=\sum_{i=k+1}^{2k}y_i\}.
\end{equation}

\medskip 

We study the upper bounds for $J_3(\gamma(X))$, where $X\subset\ZR$ is finite.
Note that the diagonal solutions already contribute to $J_3(\gamma(X))$ on the order of $|X|^3$, so the natural target is a cubic bound up to a factor of $|X|^\epsilon$ loss. 
It should be noted that at least a $\log |X|$ loss is needed when considering lattice points on the parabola. 
The sharp form of this loss may depend on $\gamma$.

\begin{theorem}
    \label{thm:main}
    For every $\epsilon>0$, there exists a constant $C_\e$  such that 
    for every finite $X\subset \ZR$ and each strictly convex $\gamma$ we have
    \begin{equation}
    \label{eq:main}
        J_3(\gamma(X))\leq C_\e|X|^{3+\epsilon}.
    \end{equation}
\end{theorem}
The implicit constant in \eqref{eq:main} is independent of $\gamma$. 
This may at first come as a surprise to researchers in harmonic analysis. 
However, unlike the typical work in multi-scale analysis, our discrete arguments make only qualitative, rather than quantitative, use of curvature. 
While $\gamma$-independence is a pleasant artifact of our argument, it also proves essential in our application to convex sequences.  
See Section \ref{section-sumset-convex}.

The best-known exponent of $|X|$ in \eqref{eq:main} prior to our work was $7/2$. This was proved in \cite{Bombieri-Bourgain} for the circle, in \cite{Bourgain-Demeter} for the parabola, and in \cite{Mudgal} for a wider class of curves. All these arguments rely crucially on variants of the  Szemer\'edi-Trotter incidence theorem.

\medskip

When the points in $X$ are $\delta$-separated, for $f\in C^2$ with $|f'|=O(1)$ and $|f''|\sim1$, \eqref{eq:main} was proved in \cite{Bourgain-Demeter} with an implicit constant that involves an additional $\delta^{-\e}$ dependence. 
This was a particular case of a more general estimate proved in \cite{Bourgain-Demeter}, called Fourier decoupling. 
The proof of the latter involved multi-scale analysis in both the spatial and frequency domains. 
The uncertainty principle and the bilinear Kakeya inequality played a crucial role. 
The key scale parameter in \cite{Bourgain-Demeter} was the diameter of arcs on the curve. 
In our argument for Theorem \ref{thm:main}, the role of the `scale parameter' is played by the number of points of $\gamma(X)$ contained in the arc. 
See \eqref{partition}.     

The point of the present proof is different: the interlacing geometry of equal triple sums is turned into a recurrence, and that recurrence is iterated all the way to the near-cubic exponent. 
Order, rather than scale or incidence geometry, is the key input of the new argument (see \eqref{order}). 
There are some vague similarities between our argument here and that in \cite{Bourgain-Demeter}, and these include the use of $L^2$ orthogonality and bootstrapping.  

\medskip 

The generality of \eqref{eq:main} may come as a surprise, at least for two different reasons.  
First, it was observed in \cite{Bourgain-Demeter} that Theorem \ref{thm:main} for the circle would be a consequence of the Unit Distance Conjecture. 
The recent disproof of the latter has cast doubt on the validity of the former.

Second, prior to the decoupling approach introduced in \cite{Bourgain-Demeter}, it was believed that proving \eqref{eq:main} when $\gamma(X)$ consists of well-spaced lattice points must involve some number-theoretic input. 
One example to have in mind is that of the lattice points on the parabola. 
The purely Fourier analytic approach in \cite{Bourgain-Demeter} showed that such input is in fact not necessary. 
However, the applicability of decoupling is restricted to the case when points in $X$ are ``well-spaced". 
In the context of lattice points, this roughly means that the diameter 
\[ 
    \operatorname{diam}(X):=\max_{x,y\in X}\dist(x,y)
\]
is at most $|X|^{O(1)}$. 
Indeed, after rescaling $X$ to lie inside $[0,1]$, the new points remain at least $\delta=\operatorname{diam}(X)^{-1}$-separated. 
Since decoupling produces a $\delta^{-\e}$-loss, well-spacedness guarantees this loss is within the acceptable $O(|X|^\e)$ margin of error. This argument is efficient if one considers consecutive lattice points on the parabola, but it loses its strength for sparse subsets. 
See Section \ref{Quadratic-Diameter-VMVT-section}.

In the current paper, the well-spaced restriction is removed.
This is crucial to our second application, as the lattice points $P_m$ on the circle  $\sqrt{m}\ZS^1$ are not well-spaced for certain values of $m$. Indeed, it is known that $|P_m|\ll_\e m^\e$, and that $\operatorname{diam}(P_m)\sim \sqrt{m}$ for infinitely many $m$. For comparison, 
Corollary \ref{thm:maincircle} delivers the estimate $J_3(P_m)\ll_\e |P_m|^{3+\e}$, stronger than the estimate $J_3(P_m)\ll_\e |P_m|^{3}m^{\e}$ given by decoupling. 
See Section \ref{section-energy-circles}.

\medskip 

Theorem \ref{thm:main} may be viewed as the beginning of a potentially broad new program, geared towards obtaining exponential sum estimates beyond the reach of decoupling. 
There are multiple avenues to explore, and the potential for further applications to questions in number theory and spectral theory is significant.

\bigskip 

\subsection{Applications}
We present three applications of our result.    

\medskip 

\subsubsection{Quadratic diameter-free Vinogradov system}
\label{Quadratic-Diameter-VMVT-section} 

We recall the following conjecture, stated explicitly in \cite{Mudgal-2020, Wooley-diameter}.

\begin{conjecture}
\label{diameter-free-VMVT}
Let $n, k\geq 2$ be integers.
Suppose that $\Gamma(t)$ is the moment curve $(t, t^2, \ldots, t^n)$ and $X\subset\ZZ$ is finite with $|X|\geq2$. 
Then
\begin{equation}
\label{diameter-free-esti}
    J_k(\Gamma(X))
    \ll_\epsilon |X|^{k+\e}+
    |X|^{2k-\frac{n(n+1)}{2}+\epsilon}.
\end{equation}
\end{conjecture}
When $X=\{1,\ldots,|X|\}$, Conjecture \ref{diameter-free-VMVT} reduces to the Vinogradov mean value theorem, which was proved by two different methods, efficient congruencing and decoupling, in \cite{Wooley-cubic, Wooley-VMVT} and \cite{BDG} respectively. 
In fact, both methods yield the stronger estimate
\[
    J_k(\Gamma(X))
    \ll_\epsilon
    \operatorname{diam}(X)^\epsilon |X|^k
    +
    \operatorname{diam}(X)^\epsilon |X|^{2k-\frac{n(n+1)}{2}}.
\] 
Clearly $|X|\ll \operatorname{diam}(X)$, and since there is no loss of a factor of $\operatorname{diam}(X)^\varepsilon$ in \eqref{diameter-free-esti}, Conjecture \ref{diameter-free-VMVT} may be referred to as the \emph{diameter-free Vinogradov mean value conjecture}.    

When $n=2$, Mudgal \cite{Mudgal} used sophisticated tools from additive combinatorics to prove that $J_k(\Gamma(X))\ll |X|^{2k-3+c_k}$ for some $c_k>0$.
This improves the exponent obtained by trivial interpolation using the three-fold energy result $J_3(\Gamma(X))\ll|X|^{7/2}$ and makes progress towards Conjecture \ref{diameter-free-VMVT}.
Wooley \cite{Wooley-diameter}, on the other hand, investigated this conjecture by attempting to transform the set $X$ into another set with diameter $|X|^{O(1)}$, while preserving the solution count in 
\eqref{diameter-free-esti}.

\medskip  

Via interpolation, Theorem \ref{thm:main} implies that  
\begin{equation}
    \label{Jk}
    J_k(\gamma(X))
    \ll_\epsilon
    |X|^{2k-3+\epsilon},
    \qquad k\geq 4.
\end{equation} 
When specialized to the parabola, this confirms Conjecture \ref{diameter-free-VMVT} when $n=2$ and its more general form in \cite[Question 2.13]{Bourgain-Demeter}.
Moreover, at least two families of lattice examples show that \eqref{Jk} is sharp:
\begin{enumerate}
    \item Take $X=[1,N]\cap\ZZ$ and $\gamma=(t, t^2)$, so that every $k$-fold sum of points of $\gamma(X)$ is a lattice point in the rectangle $[k,kN]\times[k,kN^2]$.
    \item Take $\gamma$ as the Jarn\'ik curve for the square $[1,N]^2$, a strictly convex curve that contains $\gg N^{2/3}$ lattice points in $[1,N]^2$, and $X$ as the projection of $\gamma\cap [1,N]^2$ onto the $x$-axis, so that every $k$-fold sum of points of $\gamma(X)$ is a lattice point in the square $[k,kN]^2$.
    A generalization of this example can be found in \cite[Section 6]{Cairo-Zhang}.
\end{enumerate}

\medskip  

\subsubsection{Three-fold additive energy for lattice points on circles}
\label{section-energy-circles}

We continue with another immediate corollary.
\begin{corollary}
\label{thm:maincircle}	

For each finite $P\subset \ZS^1$ and $\epsilon>0$ we have 
\[
    J_3(P)\ll_\epsilon |P|^{3+\epsilon}.
\]
In particular, letting $P_m$ denote the lattice points on the circle $\sqrt{m}\ZS^1$ of radius $\sqrt{m}$, $m\in \ZN$, we have
\begin{equation}
\label{eq:maincircle}
    J_3(P_m)\ll_\epsilon |P_m|^{3+\epsilon}.
\end{equation}	
\end{corollary}
\begin{proof}
For two finite, pairwise disjoint sets $P_1,P_2$ in the plane we have
\[
    J_3(P_1\cup P_2)\le (J_3(P_1)^{1/6}+J_3(P_2)^{1/6})^6.
\]
One way to see this is via the identity
\[
    J_3(P)=\lim_{R\to\infty}\frac1{R^2}\int_{[0,R]^2}|\sum_{\xi\in P}e^{2\pi ix\cdot \xi}|^6dx,
\]
and the triangle inequality in $L^6$. We apply this to $P_1=P\cap \ZS^1_+$ and $P_2=P\cap \ZS^1_-$, where $\ZS^{1}_{\pm}$ are the strictly convex/concave lower/upper halves of $\ZS^1$.
\end{proof}

Understanding $J_3(P_m)$ is connected with the study of arithmetic random waves; see \cite{Krishnapur-Kurlberg-Wigman} and \cite{Bombieri-Bourgain}. 
The first non-trivial upper bound $o(|P_m|^{4})$ for $J_3(P_m)$ was obtained by Bourgain, with the result appearing in \cite[Theorem 2.2]{Krishnapur-Kurlberg-Wigman}. 
The proof used tools from additive combinatorics. 
The estimate \eqref{eq:maincircle} was conjectured to hold in \cite[page 710]{Krishnapur-Kurlberg-Wigman} and in \cite[page 3344]{Bombieri-Bourgain}. 
Both papers cautiously speculate on the possibility that the stronger upper bound $\ll|P_m|^3$ might hold.

Moreover, in \cite{Bombieri-Bourgain}, many specialized unconditional and conditional estimates were obtained for $J_3(P_m)$, using deep number theory.
We mention two such examples. 
First, their Theorem 8 proves that if \eqref{eq:maincircle} fails along a special sequence of $m$, this would imply the existence of certain elliptic curves of unbounded rank.
Second, in \cite[Theorem 25]{Bombieri-Bourgain},  \eqref{eq:maincircle} is proved for a random set of (squared) radii $m$ lying inside a specialized set of square-free numbers, but only conditional on the Riemann hypothesis and the Birch and Swinnerton--Dyer conjecture for the L-functions of elliptic curves over $\ZQ$. 
This was later proved unconditionally in \cite{Bourgain-Li} using a decoupling estimate with a sharper constant. 
However, this proof relies crucially on the fact that, under the special restriction on $m$, the size $|P_m|$ is close to the maximum value allowed by the divisor bound. 
For comparison, Theorem \ref{thm:main} works for all values of $|P_m|$.

\medskip

\subsubsection{Additive behavior for convex sequences}
\label{section-sumset-convex}

Our third application of Theorem \ref{thm:main} improves the best-known bounds on the additive behavior of convex sequences.

\begin{definition}
    Given $N>1$, we say a finite sequence of numbers (identified as a set) $A=\{a_1<a_2<\cdots<a_N\}\subset\ZR$ is \emph{convex} if 
    \[
    a_{k+1}-a_k>a_k-a_{k-1} \text{ for all $k\in[2,N-1]\cap \ZZ$}.
    \]
\end{definition}

An old conjecture attributed to Erd\H{o}s, but first written down in Hegyv\'ari's paper \cite{Heg86}, is that convex sequences should not have additive structure, in the sense that they should have nearly maximal sum and difference set.

\begin{conjecture}
    Let \(A \subset \mathbb{R}\) be a finite convex sequence. Then
    \[
    \min\{|A+A|, |A-A|\} \gg_{\e} |A|^{2 - \e}
    .\]
\end{conjecture}

We improve on the previous result by the first author \cite{Cushman-sumproduct}, who obtained the exponent $\frac{8}{5}+\frac{1}{3440}$ for $A-A$ and $\frac{46}{29}$ for $A+A$. In particular, we prove the following theorem in Section \ref{section-application}.

\begin{theorem}
    \label{difference-set-thm}
    Let $N\in[2,\infty)\cap \ZZ$ and let $A$ be a convex sequence of cardinality $N$.
    Then for any $\e>0$, there exists $c_\e>0$ independent of $A$ such that
    \begin{equation}
        \label{difference-set}
        |A-A|\geq c_\e|A|^{5/3-\e}, \qquad |A+A|\geq c_\e|A|^{8/5-\e}.
    \end{equation}
    
\end{theorem}

\bigskip  

\subsection{Key ideas and outline of the argument}
\label{outline-subsection}

With $N:=|X|$, we order the set
\[
X=\{t_1<\cdots<t_N\}.
\]
For $u=t_i$ and $v=t_j$, we introduce the rank and gap
\[
\operatorname{rk}_X(u)=i,\qquad \gap_X(u,v):=|i-j|.
\]

\smallskip 

The key idea behind Theorem \ref{thm:main} is that, given a point $p\in\ZR^2$, there exists an `ordered structure' for the set of triples $(x,y,z)\in X^3$ satisfying
\begin{equation}
    \nonumber
    \gamma(x)+\gamma(y)+\gamma(z)=p.
\end{equation}
More precisely, suppose we have two triples $(x,y,z)$ and $(x',y',z')$ listed as $x<y<z$ and $x'<y'<z'$, such that 
\begin{equation}
    \label{six-tuples}
    \gamma(x)+\gamma(y)+\gamma(z)=\gamma(x')+\gamma(y')+\gamma(z').
\end{equation}
If $x<x'$, then these six numbers must satisfy
\begin{equation}
    \label{order}
    x<x'<y'<y<z<z'.
\end{equation}
The proof of this ordering property can be found in Lemma \ref{interlacing-lem}, which is a consequence of Karamata's inequality (see Lemma \ref{karamata}). 

Thus, the set of triples 
\begin{equation}
    \nonumber
    \Omega(p):=\{(x_j,y_j,z_j)\in X^3:x_j<y_j<z_j,\, \gamma(x_j)+\gamma(y_j)+\gamma(z_j)=p\}
\end{equation}
can be organized as (with $n=|\Omega(p)|$)
\begin{equation}
    \label{interlacing-property}
    x_1<x_2<\cdots<x_n<y_n<y_{n-1}<\cdots<y_1<z_1<z_2<\cdots<z_n.
\end{equation}
This interlacing property significantly reduces the complexity of the set $\Omega(p)$, and it is the key input in the proof of Theorem \ref{thm:main}.
In fact, with the interlacing property, a rather simple application of the Szemer\'edi--Trotter-type theorem for translations of one strictly convex curve leads to $J_3(\gamma(X))\ll |X|^{10/3}$. 
We will provide a sketch of the argument at the end of this subsection.

\begin{definition}
    \label{Th}
    Let $1\leq D\leq N$.
    Define the quantity $T_X(D)$ as the number of six-tuples
    \begin{equation}
        \label{condition-1}
        x < x' < y' < y < z < z',
        \qquad
        x,x',y,y',z,z' \in X,
    \end{equation}
    such that $ \gamma(x)+\gamma(y)+\gamma(z)=\gamma(x')+\gamma(y')+\gamma(z')$, and
    \begin{equation}
    \nonumber
        \operatorname{gap}_X(x,x'),\
        \operatorname{gap}_X(y,y'), \
        \operatorname{gap}_X(z,z') \leq D.
    \end{equation}
\end{definition}

\begin{definition}
    \label{extremal-def}
    Define the extremal quantities
    \begin{equation}
        \nonumber 
        J(N)=\sup_{\substack{X\subset \R,\\ |X|=N}}J_3(\gamma(X)),
        \qquad
        T(N,D)=\sup_{\substack{X\subset \R,\\ |X|=N}}T_X(D).
    \end{equation}
\end{definition}

\medskip 

A key consequence of the interlacing property \eqref{interlacing-property} is (see Proposition \ref{telescoping-prop})
\begin{equation}
\label{telescoping-intro}
    J(N)\ll \frac{N^4}{D} + \frac{N}{D}T(N,D).
\end{equation}
This is our first key estimate.

Our second key estimate is   (see Proposition \ref{iteration-prop}): for each $\eta>0$ and $1\leq D\leq N$,
\begin{equation}
\label{iteration-intro}
    T(N,D)\ll_{\eta}\left({\frac ND}\right)^{2+\eta}J(D).
\end{equation}
These two estimates combine to give 
\begin{equation}
\nonumber
    J(N)\ll_\eta \frac{N^4}{D} +\left({\frac ND}\right)^{3+\eta}J(D).
\end{equation}
Oversimplifying a bit, if we assume polynomial growth, meaning $J(M)\sim M^\alpha$ for each $M$, we immediately get $\alpha\le 3$, so in fact $\alpha=3$.

We remark that our proof of \eqref{telescoping-intro} is purely combinatorial. 
On the other hand, the proof of  \eqref{iteration-intro} uses harmonic analysis: the classical boundedness of the Hilbert transform, $L^2$ orthogonality and vector-valued interpolation. 

Strict convexity plays two critical roles in our argument. 
First, we use it to obtain the interlacing \eqref{interlacing-property}, leading to \eqref{telescoping-intro}. 
Second, we use it to derive the geometric observation behind the use of $L^2$ orthogonality; see Lemma \ref{cone-for-difference}.

\bigskip  

Finally, let us show $J_3(\gamma(X))\ll N^{10/3}$.
Write $|X|=N$. 
For any point $p\in\ZR^2$, denote by
\begin{align}
    \nonumber
    \ell(p)
    &=
    \#\{
    (x,x',y)\text{ obeying \eqref{condition-1} and }\operatorname{gap}_X(x,x')\le D:\gamma(x)-\gamma(x')+\gamma(y)=p 
    \},\\ \nonumber
    r(p)
    &=
    \#\{
    (y',z,z')\text{ obeying \eqref{condition-1} and } \operatorname{gap}_X(z,z')\le D:\gamma(y')-\gamma(z)+\gamma(z')=p \}.
\end{align}
We have $T_X(D)\le \sum_{p}\ell(p)r(p)$, by ignoring the gap restriction in the middle term.

Consider two families of translations of the strictly convex curve $\gamma$: $\{\gamma+\gamma(x)-\gamma(x'): x, x'\in X, \ \operatorname{gap}_X(x,x') \leq D\}$ and $\{\gamma+\gamma(z')-\gamma(z): z, z'\in X, \ \operatorname{gap}_X(z,z') \leq D\}$, and let $\cC$ be the union of these two families.
Observe that 
\begin{equation}
\label{curve-upperbound-intro}
    |\cC|\ll ND.
\end{equation}
We note that $\ell(p)+r(p)$ is bounded by the number $I(p)$ of incidences between the point $p$ and the curves $\cC$. 
By Lemma \ref{lem:pair-sum} we have that $I(p)\ll N$. 
For each dyadic $r\ll N$ let $Q_r$ be the number of those $p$ with $I(p)\sim r$.
Using the  Szemer\'edi-Trotter type theorem for translations of a fixed convex curve (see, for example, \cite[Theorem 6]{Szekely}), we have
\[
Q_r\ll \frac{|\cC|}{r}+\frac{|\cC|^2}{r^3}.
\]
This and \eqref{curve-upperbound-intro} gives
\begin{equation}
    \nonumber
    \sum_{p:\;I(p)\ge \sqrt{D}}\ell(p)r(p)\le \sum_{p:\;I(p)\ge \sqrt{D}}(\ell(p)+r(p))^2\ll \sum_{ \sqrt{D}\ll r\ll N}r^2(\frac{|\cC|}{r}+\frac{|\cC|^2}{r^3})\ll N^2D^{3/2}.
\end{equation}
On the other hand, we have the density estimate
\begin{equation}
    \nonumber
    \sum_{p}\ell(p), \ \sum_{p}r(p)\ll N^2D.
\end{equation}
This immediately gives the same estimate as before 
\[
\sum_{p:\;I(p)\le \sqrt{D}}\ell(p)r(p)\ll N^2D^{3/2}.
\]
Thus, $T_X(D)\ll N^2D^{3/2}$.
We plug this back into \eqref{telescoping-intro} and choose $D=N^{2/3}$ to conclude that $J_3(\gamma(X))\ll |X|^{10/3}$.

\bigskip

\subsection{Communication with AI}
The use of GPT-5.6 was central to the main theorem.
In the first stage, we only considered the case when $\gamma$ is the parabola.
We provided the AI with known results and some of our related attempts, including earlier improvements over the 4-fold energy $J_4(\gamma(X))$ (private conversation with AI), and asked it to improve upon the $7/2$ bound for $J_3(\gamma(X))$.
The AI then observed that, for any triple $(x_1, x_2, x_3)$ satisfying
\begin{equation}
    \label{set-of-pts}
    \gamma(x_1)+\gamma(x_2)+\gamma(x_3)=(p_1,p_2),
\end{equation}
the numbers $x_1, x_2, x_3$ are precisely the roots of a cubic polynomial of the form
\[
    P_{q_1,q_2}(x)=x^3+q_2x^2+q_1x+q_0,
\]
where $q_2$ and $q_1$ are uniquely determined by $(p_1,p_2)$.

Consequently, for each fixed point $p=(p_1,p_2)$, there is a one-to-one correspondence between the triples $(x_1,x_2,x_3)$ with $x_1<x_2<x_3$ satisfying \eqref{set-of-pts} and the corresponding values of the constant term $q_0=q_0(x_1,x_2,x_3)$.
Geometrically, $q_0$ is the vertical intercept of the graph of $P_{q_1,q_2}(x)$. 
As $q_0$ increases, $x_1$ must decrease, $x_2$ must increase, and $x_3$ must decrease. 
This is exactly the interlacing phenomenon described in \eqref{interlacing-property}, which leads to the $10/3$ exponent for $J_3(\gamma(X))$, outlined in Section \ref{outline-subsection}.

After the successful advances in the parabolic case, we immediately realized that the same idea extends to strictly convex curves. We also found that this idea iterates nicely, and determined that this iteration should form the main scheme in the second stage.
This iteration, however, differed from the formula \eqref{iteration-intro} used in the current argument.
It repeatedly used the interlacing argument to bound $T_X(D)$ via variants of  \eqref{telescoping-intro}.
From the beginning of this stage, we experimented with several types of iteration formulas, including formulas involving a substantially more complicated form of $T_X(D)$, in which more than three parameters were used to measure the gaps between the paired elements of the six-tuples in \eqref{six-tuples}.
Although this approach yielded further improvements, it continued to rely on the Szemer\'edi--Trotter-type estimate, and the resulting sequence of exponents appeared to converge to a value just slightly below $16/5$.

The final stage began when we instructed the AI to prove a proposed iteration formula that was even more complicated than \eqref{iteration-intro}.
It successfully established this formula, at which point we realized that its simplest version, namely \eqref{iteration-intro}, was already strong enough to prove the near-optimal bound for $J_3(\gamma(X))$.

\bigskip

\noindent {\bf Notation.}
We write $A\ll B$ to mean that $|A|\leq C B$ for some constant $C>0$. 
A subscript indicates the parameters on which the implicit constant may depend; for example, $A\ll_{\epsilon} B$ means that $C$ may depend on $\epsilon$.
Unless otherwise stated, all implicit constants are absolute.

\bigskip

\noindent {\bf Acknowledgements.}
The second author is partially supported by the NSF grant DMS-2349828.
The third author is partially supported by the NSF grant DMS-2453583.

\bigskip

\section{Preliminary results}

Fix $1\leq D\leq N$ and partition $X$ into consecutive rank blocks
\begin{equation}
    \label{partition}
    I_1,\ldots,I_h,
\end{equation}
of size $D$, except possibly for the final block, which is allowed to be shorter. 
Then
\begin{equation}
    \nonumber
    h=\left\lceil\frac ND\right\rceil\leq\frac{2N}{D}.
\end{equation}

\bigskip 

\subsection{Elementary convex geometry}
\begin{lemma}
    \label{lem:pair-sum}
    Suppose $u\leq u'$ and $v\leq v'$.
    Then
    \begin{equation}
    \nonumber
        \gamma(u)-\gamma(u')=\gamma(v)-\gamma(v')
    \end{equation}
    if and only if $u=u'$ and $v=v'$, or $u=v$ and $u'=v'$.
\end{lemma}
\begin{proof}
    Recall that $\gamma(x)=(x,f(x))$.
    Comparing the first coordinates gives
    \[
    u'-u=v'-v=:h\geq 0.
    \]
    If \(h=0\), then \( (u,v)=(u',v')\).
    If \(h>0\), comparing the second coordinates gives
    \[
    f(u+h)-f(u)=f(v+h)-f(v).
    \]
    Since \(f\) is strictly convex, for every fixed \(h>0\), the function $x\longmapsto f(x+h)-f(x)$ is strictly increasing. 
    Hence \(u=v\), and consequently
    \[
    u'=u+h=v+h=v'.
    \]
    The converse is immediate.
\end{proof}

\begin{lemma}
    \label{cone-for-difference}
    Let $\{I_j\}$ be the partition of $X$ in \eqref{partition}.
    For each $I_j$, define the set of chord-vectors
    \begin{equation}
        \nonumber
        E_j=\{\gamma(x')-\gamma(x):x,x'\in I_j,\ x<x'\}.
    \end{equation}
    Then $E_j$ are pairwise disjoint.
    In fact, there exist pairwise disjoint conic regions $V_j$ (two-sided, symmetric around the origin) such that $E_j$ is contained in the interior of $V_j$, and the $V_j$ are adjacent to each other.
\end{lemma}
\begin{proof}
    For $u<v$, write
    \begin{equation}
        \nonumber
        \sigma(u,v)=\frac{f(v)-f(u)}{v-u}.
    \end{equation}
    Since $f$ is strictly convex, $\sigma(u,v)$ is strictly increasing in each variable separately. 
    
    It follows that the slopes of the chord-vectors $\gamma(v)-\gamma(u): u,v\in I_j$  are strictly smaller than the slopes of  $\gamma(v)-\gamma(u): u,v\in I_k$, if $j<k$. 
    Hence these vectors can be placed inside pairwise disjoint open cones
    \begin{equation}
    \nonumber
        V_1,\ldots, V_h,
    \end{equation} 
    such that the adjacent cones $V_j$ and $V_{j+1}$ share a common boundary.
    The following figure illustrates two such disjoint cones.   \qedhere
    
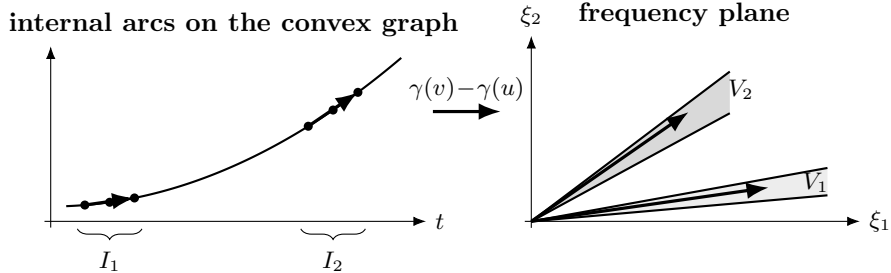
    \begin{figure}[htbp]
        \centering
        \begin{tikzpicture}[>=Latex,font=\small]
            \begin{scope}[xshift=0cm,yshift=0cm,x=0.82cm,y=0.82cm]
                \node[font=\bfseries] at (2.95,3.18) {internal arcs on the convex graph};
                \draw[->] (-0.1,0) -- (6.05,0) node[right] {$t$};
                \draw[->] (0,-0.1) -- (0,2.82);
                \draw[thick,domain=0.25:5.65,samples=120] plot (\x,{0.075*\x*\x+0.24});
                
                \coordinate (a1) at (0.55,{0.075*0.55*0.55+0.24});
                \coordinate (a2) at (0.95,{0.075*0.95*0.95+0.24});
                \coordinate (a3) at (1.35,{0.075*1.35*1.35+0.24});
                \coordinate (b1) at (4.15,{0.075*4.15*4.15+0.24});
                \coordinate (b2) at (4.55,{0.075*4.55*4.55+0.24});
                \coordinate (b3) at (4.95,{0.075*4.95*4.95+0.24});
                
                \foreach \P in {a1,a2,a3,b1,b2,b3} \fill (\P) circle (1.7pt);
                
                \draw[->,very thick] (a1) -- (a3);
                \draw[->,very thick] (b1) -- (b3);
                \draw[decorate,decoration={brace,mirror,amplitude=4pt}] (0.43,-0.12) -- (1.47,-0.12)
                node[midway,below=5pt] {$I_1$};
                \draw[decorate,decoration={brace,mirror,amplitude=4pt}] (4.03,-0.12) -- (5.07,-0.12)
                node[midway,below=5pt] {$I_2$};
            \end{scope}
            
            \draw[->,very thick] (5.05,1.45) -- (5.95,1.45)
            node[midway,above] {$\gamma(v)\!-\!\gamma(u)$};
            
            \begin{scope}[xshift=6.35cm,yshift=0cm,x=0.86cm,y=0.86cm]
                \node[font=\bfseries] at (2.35,3.18) {frequency plane};
                \draw[->] (-0.1,0) -- (5.05,0) node[right] {$\xi_1$};
                \draw[->] (0,-0.1) -- (0,2.85) node[above] {$\xi_2$};
                
                \path[fill=gray!14] (0,0) -- (4.55,0.40) -- (4.55,0.82) -- cycle;
                \draw[thick] (0,0) -- (4.55,0.40);
                \draw[thick] (0,0) -- (4.55,0.82);
                \draw[->,very thick] (0,0) -- (3.65,0.52);
                \node[anchor=west] at (4.03,0.60) {$V_1$};
                
                \path[fill=gray!30] (0,0) -- (3.05,1.66) -- (3.05,2.30) -- cycle;
                \draw[thick] (0,0) -- (3.05,1.66);
                \draw[thick] (0,0) -- (3.05,2.30);
                \draw[->,very thick] (0,0) -- (2.42,1.68);
                \node[anchor=west] at (2.88,2.03) {$V_2$};
                
            \end{scope}
        \end{tikzpicture}
        
        \label{fig:emerging-cones}
    \end{figure}

\end{proof}

\begin{definition}[Majorization]
    Let $a=(a_1,\dots,a_n)$ and $b=(b_1,\dots,b_n)$ be decreasing sequences of numbers.
    We say that $a$ \emph{majorizes} $b$ if 
    \begin{equation}
    \nonumber
        \sum_{i=1}^n a_i=\sum_{i=1}^n b_i \quad \text{ and} \quad\sum_{i=1}^k a_i\ge \sum_{i=1}^k b_i
        \quad\text{for }k=1,\dots,n-1.
    \end{equation}
\end{definition}

\begin{lemma}[Karamata's inequality]
    \label{karamata}
    Let $a=(a_1,\dots,a_n)$ and $b=(b_1,\dots,b_n)$ be two distinct decreasing sequences of numbers:
    \[
    a_1\ge \cdots\ge a_n,
    \qquad
    b_1\ge \cdots\ge b_n.
    \]
    Let $f$ be strictly convex.
    If $a$ majorizes $b$, then
    \[
    \sum_{i=1}^n f(a_i)>
    \sum_{i=1}^n f(b_i).
    \]
\end{lemma}

\medskip 

As mentioned in the introduction, a key observation for bounding $J_3(\gamma(X))$ is the following interlacing lemma, which follows from Karamata's inequality.

\begin{lemma}[Interlacing]
    \label{interlacing-lem}
    
    Let $(x,y,z)$ and $(x',y',z')$ be two strictly increasing triples of numbers.
    Suppose that $\gamma(x)+\gamma(y)+\gamma(z)=\gamma(x')+\gamma(y')+\gamma(z')$.
    If $x<x'$, then
    \[
    x<x'<y'<y<z<z'.
    \]
\end{lemma}

\begin{proof}
    The two coordinates in $\gamma(x)+\gamma(y)+\gamma(z)=\gamma(x')+\gamma(y')+\gamma(z')$ give
    \begin{align}
        x+y+z
        &=x'+y'+z', \label{eq:scalar-sums}\\
        f(x)+f(y)+f(z)
        &=f(x')+f(y')+f(z'). \label{eq:f-sums}
    \end{align}
    
    We first show that $z<z'$. 
    Consider the decreasing triples
    \[
    a=(z,y,x),
    \qquad
    b=(z',y',x').
    \]
    They have the same total sums by \eqref{eq:scalar-sums}. 
    Since $x<x'$, we have $z+y>z'+y'$.
    Thus, if $z\geq z'$, then $a$ majorizes $b$, and Karamata's inequality gives
    \[
    f(z)+f(y)+f(x)
    >
    f(z')+f(y')+f(x'),
    \]
    which contradicts \eqref{eq:f-sums}. 
    Hence $z<z'$.
    
    From $x<x'$ and $z<z'$, we clearly have $y'<y$.
    The lemma then follows from the fact that $(x,y,z)$ and $(x',y',z')$ are strictly ordered triples.
    \qedhere

\end{proof}

\bigskip

\subsection{The first key estimate}

Recall Definition \ref{Th} that for a natural number $D\in[1, N]$, $T_X(D)$ counts the number of six-tuples
\[
x < x' < y' < y < z < z',
\qquad
x,x',y,y',z,z' \in X,
\]
such that $\gamma(x)+\gamma(y)+\gamma(z)=\gamma(x')+\gamma(y')+\gamma(z')$, and
\[
\operatorname{gap}_X(x,x'),\
\operatorname{gap}_X(y,y'), \
\operatorname{gap}_X(z,z') \leq D.
\]

\begin{proposition} 
    \label{telescoping-prop}
    For every integer $1\leq D\leq N$,
    \begin{equation}
        \nonumber
        J_3(\gamma(X))\ll\frac{N^4}{D}+\frac ND T_X(D).
    \end{equation}
    Consequently,
    \begin{equation}
        \label{first-key}
        J(N)\ll\frac{N^4}{D}+\frac ND T(N,D).
    \end{equation}
\end{proposition}

\begin{proof}
    Recall that 
    \begin{equation}
        \nonumber
        \Omega(p)=\{(x_j,y_j,z_j)\in X^3:x_j<y_j<z_j,\, \gamma(x_j)+\gamma(y_j)+\gamma(z_j)=p\}.
    \end{equation}
    Let $\nu(p)$ count the number of triples $(a_1, a_2, a_3)\in\gamma(X)^3$ such that $a_1+a_2+a_3=p$.
    Then clearly we have
    \begin{equation}
    \nonumber
        \nu(p)\ll|\Omega(p)|+\rho(p),
    \end{equation}
    where $\rho(p)$ counts the number of triples $(a_1, a_2, a_3)\in\gamma(X)^3$ with repeated entries such that $a_1+a_2+a_3=p$.
    Thus, we have
    \begin{equation}
        \label{J3-twoterms}
        J_3(\gamma(X))=\sum_p\nu(p)^2\ll\sum_p|\Omega(p)|^2+\sum_p\rho(p)^2.
    \end{equation}
    Let $p=(p_1,p_2)$.
    Up to permutations, $\rho(p)$ counts the number of triples $(x,x,y)\in X^3$ such that $2f(x)+f(y)=p_2$ and $2x+y=p_1$.
    Thus, we have
    \begin{equation}
        \label{repeated-entry}
        2f(x)+f(p_1-2x)=p_2.
    \end{equation}
    Note that, as a function of $x$, $2f(x)+f(p_1-2x)$ is also strictly convex.
    Thus, there are at most two solutions to \eqref{repeated-entry}, which shows that $\rho(p)\ll1$.
    Since $\sum\rho(p)\ll N^2$, we obtain
    \begin{equation}
        \label{repeated-estimate}
        \sum_p\rho(p)^2\ll N^2.
    \end{equation}
    
    \medskip 
    
    Regarding the first term on the right-hand side of \eqref{J3-twoterms}, let \(C > 0\) be a sufficiently large  absolute constant to be chosen momentarily, and define 
    \begin{equation}
        \nonumber
        E_+:=\{p: |\Omega(p)|\geq C(N/D)\}, \qquad E_-:=\{p: |\Omega(p)|< C(N/D)\}.
    \end{equation}
    For each $p\in E_+$, by Lemma \ref{interlacing-lem}, we can organize the set $\Omega(p)$ as
    \[
    x_1<\cdots <x_n<y_n<y_{n-1}<\cdots<y_1<z_1<z_2<\cdots<z_n,\quad n=|\Omega(p)|.
    \]
    Observe that for any \(1 \leq k < n \) we have
    \[
    \sum _{j = 1}^{n - k} \left( \operatorname{rk}_X(x_{j+k})
    - \operatorname{rk}_X(x_j) \right) \leq k N
    ,\]
    and thus,  for each level \(L \geq 3\) we have
    \[
        \sum _{k = L+1} ^{2L} \sum _{j = 1}^{n - k} \left( \operatorname{rk}_X(x_{j+k})- \operatorname{rk}_X(x_j) \right) \leq 2 L^2 N
    .\]
    Therefore the number of pairs of triples \((x_j,y_j,z_j)\) and \((x_{j+k},y_{j+k},z_{j+k})\) with \(L+1 \leq k \leq 2L \) and  \( \operatorname{gap}_X(x_j,x_{j+k}) \geq D\) is at most \(2 L^2 N / D\). Repeating the argument for \(y,z\), we find that there are at most \(6 L^2 N / D\) many \(k\)-separated triples, for which at least one of the rank gaps exceeds \(D\), for \(L+1 \leq k \leq 2L\). 
    
    Let \(M\) be the number of pairs of such \(k\)-separated triples with no rank gap restrictions and $L+1\leq k\leq 2L$. 
    Note that \(nL/2\leq  M \leq nL\).
    Take
    \[
    L  =\Big\lfloor\frac{nD}{100N}\Big\rfloor 
    \]
    so that
    \[
    6L ^2 N /D \leq M/2,
    \]
    and then choose \(C =300\) so that $n\geq 300(N/D)$ and hence \(L \geq 3\).
    
    \smallskip 
    
    With these appropriate constants, we conclude that there are
    \[
    \geq M/2 \gg n ^2 \cdot D/N  
    \]
    pairs of triples \((x_j,y_j,z_j)\) and \((x_{j+k},y_{j+k},z_{j+k})\) such that
    \[
    \operatorname{gap}_X(x_j, x_{j+k}),\quad
    \operatorname{gap}_X(y_j, y_{j+k}), \quad
    \operatorname{gap}_X(z_j,z_{j+k}) \leq D.
    \]
    Consequently, since $|\Omega(p)|^2=n^2$, we have
    \begin{equation}
    \nonumber
        \sum_{p\in E_+}|\Omega(p)|^2\ll \frac{N}{D}T_X(D).
    \end{equation}
    Finally, we also have
    \begin{equation}
    \nonumber
        \sum_{p\in E_-}|\Omega(p)|^2\ll\frac{N}{D}\sum_{p\in E_-}|\Omega(p)| \leq \frac{N^4}{D}.
    \end{equation}
    Plug these two estimates and \eqref{repeated-estimate} back into \eqref{J3-twoterms} to conclude the proof.
    \qedhere
    
\end{proof}

\bigskip

\subsection{Cone projections and $l^q(L^q)$ interpolation}
We use the convention
\begin{equation}
    \nonumber
    \widehat g(\xi)=\int_{\R^2}g(u)e^{-2\pi i u\cdot\xi}\,du.
\end{equation}

\medskip 

Let $H$ be the one-dimensional Hilbert transform defined as $\wh {Hg}(\xi)=-i\, {\rm sgn}(\xi)\wh g(\xi)$.
We use the classical estimate
\begin{equation}
    \label{Hilbert}
    \|Hg\|_{L^p(\R)}\leq C_p\|g\|_{L^p(\R)},\qquad 1<p<\infty,
\end{equation}
where $C_p$ depends only on $p$; see Stein~\cite{Stein}.

For a nonzero linear functional $\ell$ on $\R^2$, let $P_{\ell,+}$ be the Fourier multiplier with symbol $\Id_{\{\ell(\xi)>0\}}$. 
After a rotation, this is the one-dimensional Riesz projection in one coordinate, with the other coordinate acting as a parameter. 
Thus \eqref{Hilbert} and Fubini imply that, 
uniformly in the direction of $\ell$,
\begin{equation}
    \nonumber
    \|P_{\ell,+}g\|_{p}\leq C_p\|g\|_{p},\qquad 1<p<\infty.
\end{equation}
A planar two-sided cone $V$ of aperture $\leq \pi$ consists of two pieces, each of which is the intersection of two half-planes. 
The cone projection $P_V$, whose multiplier is $\Id_V$, is thus a sum of products of two commuting half-plane projections. 
Consequently,

\begin{lemma}
    \label{cone-multi-lem}
    Let $P_V$ be a multiplier operator whose multiplier is $\Id_V$.
    Then
    \begin{equation}
        \label{eq:cone-projection}
        \|P_V g\|_{p}\leq C_p\|g\|_{p},\qquad 1<p<\infty,
    \end{equation}
    with a constant independent of the location and aperture of $V$.
\end{lemma}

\medskip

Next, we recall the following classical interpolation theorem.

\begin{theorem}[Riesz-Thorin]
    Let $(\Omega,m)$ and $(Y,\nu)$ be measure spaces, let $1\leq q_0,q_1<\infty$, and let $S$ be linear on $L^{q_0}(\Omega)+L^{q_1}(\Omega)$. 
    If
    \[
    \|Sg\|_{L^{q_i}(Y)}\leq M_i\|g\|_{L^{q_i}(\Omega)},\qquad i=0,1,
    \]
    and $0<\theta<1$ satisfies
    \[
    \frac1r=\frac{1-\theta}{q_0}+\frac\theta{q_1},
    \]
    then
    \[
    \|Sg\|_{L^r(Y)}\leq M_0^{1-\theta}M_1^\theta\|g\|_{L^r(\Omega)}.
    \]
\end{theorem}
The next result is a classical consequence; we include its proof for completeness.

\begin{lemma}[Interpolation for $l^q(L^q)$ spaces]\label{lem:synthesis-RT}
    Let $P_1,\ldots,P_K$ be linear operators defined on $L^2(\R^2)+L^p(\R^2)$, where $p>3$. 
    Suppose that
    \begin{equation}
        \label{eq:L2-synthesis}
        \Big\|\sum_{j=1}^KP_jf_j\Big\|_{2}
        \leq B_2\Big(\sum_{j=1}^K\|f_j\|_{2}^2\Big)^{1/2} 
    \end{equation}
    for every $f_j\in L^2$, and that
    \begin{equation}
        \label{eq:Lp-individual}
        \sup_{1\leq j\leq K}\|P_j\|_{L^p\to L^p}\leq B_p. 
    \end{equation}
    Set $t=\frac{p}{3(p-2)}$ so that $1/3=(1-t)/2+t/p$. 
    Then
    \begin{equation}
        \nonumber
        \Big\|\sum_{j=1}^KP_jf_j\Big\|_{3}
        \leq C_pB_2^{1-t}B_p^tK^{t(1-\frac{1}{p})}
        \Big(\sum_{j=1}^K\|f_j\|_{3}^3\Big)^{1/3}.
    \end{equation}
\end{lemma}

\begin{proof}
    Let $I_K=\{1,\ldots,K\}$ and equip
    \[
    \Omega_K=I_K\times\R^2
    \]
    with the product of counting measure and Lebesgue measure. For a sequence $F=(f_j)_{j=1}^K$, define a scalar function on $\Omega_K$ by
    \[
    (UF)(j,x)=f_j(x).
    \]
    For every $1\leq s<\infty$, Tonelli's theorem gives
    \[
    \|UF\|_{L^s(\Omega_K)}^s
    =\int_{\Omega_K}|UF|^s\,d(\#\times dx)
    =\sum_{j=1}^K\int_{\R^2}|f_j(x)|^s\,dx
    =\sum_{j=1}^K\|f_j\|_{L^s}^s.
    \]
    Thus $U$ is an isometric identification
    \[
    U:\ell_K^s(L^s(\R^2))\longrightarrow L^s(\Omega_K),
    \]
    whose inverse sends a scalar function $g(j,x)$ to the sequence of its fibers $g_j(x)=g(j,x)$.
    Define the linear operator
    \[
    TF=\sum_{j=1}^KP_jf_j
    \]
    and the scalar linear operator
    \[
    S=TU^{-1}
    \]
    on $L^2(\Omega_K)+L^p(\Omega_K)$. 
    By \eqref{eq:L2-synthesis},
    \[
    \|Sg\|_{2}\leq B_2\|g\|_{L^2(\Omega_K)}.
    \]
    At the other endpoint, the triangle inequality, \eqref{eq:Lp-individual}, and H\"older's inequality give
    \[
    \|TF\|_{p}
    \leq\sum_{j=1}^K\|P_jf_j\|_{p}
    \leq B_p\sum_{j=1}^K\|f_j\|_{p}
    \leq B_pK^{1-\frac{1}{p}}\Big(\sum_{j=1}^K\|f_j\|_{p}^p\Big)^{1/p}.
    \]
    Since $U$ is an isometry, this is
    \[
    \|Sg\|_{p}\leq B_pK^{1-\frac{1}{p}}\|g\|_{L^p(\Omega_K)}.
    \]
    Apply the scalar Riesz--Thorin theorem to $S$ between exponents $2$ and $p$. Because
    \[
    \frac13=\frac{1-t}{2}+\frac tp,
    \]
    it follows that
    \[
    \|Sg\|_{3}\leq B_2^{1-t}(B_pK^{1-\frac{1}{p}})^t\|g\|_{L^3(\Omega_K)}.
    \]
    Finally take $g=UF$. Since $S(UF)=TF$ and $U$ is an isometry,
    \[
    \Big\|\sum_{j=1}^KP_jf_j\Big\|_{3}
    \leq B_2^{1-t}B_p^tK^{t(1-1/p)}
    \Big(\sum_{j=1}^K\|f_j\|_{3}^3\Big)^{1/3}. \qedhere
    \]
\end{proof}

\medskip

We apply this to functions supported on pairwise disjoint cones. The argument only relies on $L^2$ orthogonality. 
In recent literature, this type of inequality has been referred to as {\em flat decoupling}.
The negligible loss $K^\eta$ comes from using a finite $p$.

\begin{corollary}[Disjoint-cone $l^3L^3$ flat decoupling]\label{cor:ordered-cones}
    Let $g_1,\ldots,g_K$ have Fourier supports in finitely overlapping open cones. Then, for every $\eta>0$,
    \begin{equation}
        \label{eq:ordered-cone-synthesis}
        \Big\|\sum_{j=1}^Kg_j\Big\|_{3}^3
        \ll_\eta K^{1+\eta}\sum_{j=1}^K\|g_j\|_{3}^3
        \ll_\eta K^{2+\eta}\max_j\|g_j\|_{3}^3. 
    \end{equation}
\end{corollary}

\begin{proof}
    Let $P_j$ be the projection onto the cone supporting $g_j$. Since the cones are finitely overlapping, Plancherel's inequality gives
    \[
    \Big\|\sum_jP_jf_j\Big\|_{2}^2\ll\sum_j\|P_jf_j\|_{2}^2\leq\sum_j\|f_j\|_{2}^2,
    \]
    so \eqref{eq:L2-synthesis} holds with $B_2=1$. Also $P_jg_j=g_j$, and \eqref{eq:cone-projection} gives \eqref{eq:Lp-individual} with a constant depending only on $p$. Lemma~\ref{lem:synthesis-RT} therefore yields
    \[
    \Big\|\sum_jg_j\Big\|_{3}^3
    \leq C_pK^{3t(1-\frac{1}{p})}\sum_j\|g_j\|_{3}^3.
    \]
    Here $3t(1-\frac1p)=\frac{p-1}{p-2}=1+\frac1{p-2}$.
    Choose $p$ sufficiently large that $1/(p-2)\leq\eta$. This proves the first inequality in \eqref{eq:ordered-cone-synthesis}.
\end{proof}

\bigskip  

\section{Proof of the main theorem}

We are going to pick a small positive number $\delta$, and consider the $\delta$-neighborhood of the points in $\gamma(X)$.
It should be noted that while $\delta$ is allowed to be dependent on $\gamma(X)$, its use will not lead to any loss that depends on $\delta$. The use of this parameter serves the purpose of creating an essentially compact spatial domain where all exponential sums become integrable. It accounts for the lack of periodicity of exponential sums associated with generic sets $\gamma(X)$.

\medskip

\subsection{Admissible thickening}

\begin{definition}
    \label{admissible-def}
    A number $\delta>0$ is called \emph{admissible for $X$} if the following hold:
    \begin{enumerate}
        \item The points in $X$ are $\delta$-separated.
        \item For all $a_j, b_j\in \gamma(X)$, 
        \begin{equation}
            \nonumber
            a_1+a_2+a_3=b_1+b_2+b_3
        \end{equation}
        if and only if 
        \begin{equation}
            \nonumber
            \{N_\delta(a_1)+N_\delta(a_2)+N_\delta(a_3)\}\cap \{N_\delta(b_1)+N_\delta(b_2)+N_\delta(b_3)\}\not=\varnothing.
        \end{equation}
        \item 
        We have $N_\delta(E_j)\subset V_j$ (recall Lemma \ref{cone-for-difference} for $E_j,V_j$).
    \end{enumerate}
    
\end{definition}

Clearly, for a finite $X$, there exists an admissible $\delta$. 

\medskip

Now, for any finite $X$, we assign an admissible $\delta$ to it.
Choose $\widehat\vp\in C_c^\infty(\R^2)$ with
\begin{equation}
    \nonumber
    \supp(\widehat\vp)\subset B(0,\delta/100),
\end{equation}
and normalize its inverse Fourier transform so that
\begin{equation}
    \nonumber
    \|\vp\|_{L^6(\R^2)}=1.
\end{equation}
For $Y\subset X$, define
\begin{equation}
    \label{FY}
    F_Y(u)=\vp(u)\sum_{y\in Y} e^{2\pi i\gamma(y)\cdot u},
\end{equation}
or equivalently, $\widehat F_Y(\xi)=\sum_{y\in Y}\widehat\vp(\xi-\gamma(y))$.
Formally taking $\widehat{\vp}=\delta_0$ would recover the Dirac mass, a convenient heuristic for the reader to keep in mind.

\medskip

If $\delta$ is admissible for $X$, then item (2) of Definition \ref{admissible-def} implies that, for every $Y\subset X$, we have
\begin{equation}
    \label{J3-F}
    \|F_Y\|_{L^6}^6=J_3(\gamma(Y)). 
\end{equation}
Moreover, whenever $\omega$ is a nonzero difference of two triple sums from $\gamma(X)$,
\begin{equation}
    \nonumber
    \int_{\R^2}|\vp(u)|^6 e^{2\pi i\omega\cdot u}\,du=0.
\end{equation}
For $\omega=0$, the integral equals $1$.

\medskip

Therefore, since $\delta$ is admissible for $X$, by Proposition \ref{telescoping-prop} and \eqref{J3-F}, we have
\begin{lemma}
    \label{telescoping-lemma-delta}
    For any $1\leq D\leq N$, we have 
    \[
    \|F_X\|_6^6\ll \frac{N^4}{D}+ \frac{N}{D} T_X(D).
    \]
\end{lemma}

\bigskip 

\subsection{The second key estimate}

Recall the partition of $X$ into $\{I_j\}$ in \eqref{partition}.
Define
\begin{equation}
    \nonumber
    X_j:= I_j\cup I_{j+1},
\end{equation}
with the convention that $I_{h+1}=\varnothing$.

\begin{lemma}
    \label{Sh-G-lem}
    Let $G=\sum_j(|F_{X_j}|^2-|X_j|\cdot |\vp|^2)$.
    Then
    \begin{equation}
        \label{Sh-G}
        T_X(D)\ll\int G^3.
    \end{equation}
\end{lemma}
\begin{proof}
    Let $G_j=|F_{X_j}|^2-|X_j|\cdot |\vp|^2$.
    By \eqref{FY}, we have
    \[
    G_j(u)
    =
    |\vp(u)|^2
    \sum_{\substack{a,b\in\gamma(X_j),\\a\neq b}}
    e^{2\pi i(a-b)\cdot u}.
    \]
    Therefore, we have
    \[
    \int G^3=
    \sum_{j_1,j_2,j_3}
    \#\left\{
    (a_1,a_2,a_3,b_1,b_2,b_3):
    \begin{array}{l}
        a_k\not= b_k,\quad a_k, b_k\in \gamma(X_{j_k}),\\[2mm]
        a_1+a_2+a_3=b_1+b_2+b_3
    \end{array}
    \right\}.
    \]
    Two triples contributing to $T_X(D)$ have coordinates in adjacent intervals $I_j$.
    This in particular implies \eqref{Sh-G}.
    \qedhere
    
\end{proof}

\begin{proposition}
    \label{ortho-prop}
    Let $G=\sum_j(|F_{X_j}|^2-|X_j|\cdot |\vp|^2)$.
    Recall the definition \eqref{partition} for $h$.
    Then for any $\eta>0$, there exists a $C_\eta$ such that
    \begin{equation}
        \|G\|_3^3\leq C_\eta \,h^{1+\eta}\sum_{j=1}^h\|F_{X_j}\|_6^6.
    \end{equation}
\end{proposition}
\begin{proof}
    Let $G_j=|F_{X_j}|^2-|X_j|\cdot |\vp|^2$.
    Then by Lemma \ref{cone-for-difference} and the fact that $\delta$ is admissible for $X$ (in particular, item (3) in Definition \ref{admissible-def}), we know that there exist cones $\{V_j:j=1,\ldots ,h\}$ such that
    \begin{enumerate}
        \item $\supp (\wh G_j)\subset V_j\cup V_{j+1}$.
        \item $\{V_j\setminus\{0\}:j=1,\ldots ,h\}$ are pairwise disjoint.
    \end{enumerate}
    Let $P_j$ be a multiplier operator in $\ZR^2$ whose Fourier multiplier is the characteristic function of $V_j\cup V_{j+1}$, which is also a cone.
    Since $P_j G_j=G_j$, by the triangle inequality and Corollary \ref{cor:ordered-cones}, we have
    \begin{equation}
        \nonumber
        \Big\|\sum_{j=1}^h G_j\Big\|_3^3\ll   C_\eta\, h^{1+\eta}\sum_{j=1}^h\|G_j\|_3^3.
    \end{equation}
    Since $\|G_j\|_3\ll \|F_{X_j}\|_6^2+|X_j|\ll \|F_{X_j}\|_6^2$, this concludes the proof.
    \qedhere
    
\end{proof}

\medskip

Finally, we establish our second main estimate.

\begin{proposition}
    \label{iteration-prop}
    For every $\eta>0$, there exists a constant $C_\eta$ such that for every finite $X$ with $|X|=N$, and every $1\leq D\leq N$,
    \begin{equation}
        \nonumber
        T_X(D)\leq C_\eta\left(\frac{N}{D}\right)^{2+\eta}J(D).
    \end{equation}
    Consequently,
    \begin{equation}
        \label{second-key}
        T(N,D)\leq C_\eta\left(\frac{N}{D}\right)^{2+\eta}J(D).
    \end{equation}
\end{proposition}
\begin{proof}
Note that $\#j=h\ll N/D$.
By Proposition \ref{ortho-prop} and Lemma \ref{Sh-G-lem}, 
\begin{equation}
    \nonumber
    T_X(D)\ll_\eta \left(\frac{N}{D}\right)^{1+\eta}\sum_j\|F_{X_j}\|_6^6.
\end{equation}
Since $|X_j|\leq 2D$, this proves the proposition by \eqref{J3-F} and the triangle inequality.
\qedhere
    
\end{proof}

\bigskip

\subsection{Proof of the main theorem}

Recall \eqref{first-key} and \eqref{second-key}, which give
\begin{equation}
    \label{first-key-2}
    J(N)\ll_\eta\frac{N^4}{D}+\left(\frac ND\right)^{3+\eta} J(D).
\end{equation}

\begin{proof}[Proof of Theorem \ref{thm:main}]
    
    Let us assume that for some $s>1$ we have
    \begin{equation}
        \label{ehdgefurgfhuir}
        J(N)\ll N^{2+s}
    \end{equation}
    for each $N\ge 1$ and each $\eta>0$. Then choosing $D\sim N^{\frac1s}$ in \eqref{first-key-2} gives, for each $N$
    \[
    J(N)\ll_\eta N^{4-\frac1{s}+\eta}.
    \]
    Since we know \eqref{ehdgefurgfhuir} holds for $s=2$, we may iterate this procedure to find successive values of $s$ that also work (at the expense of larger, but harmless implicit constants): $2\mapsto 3/2\mapsto 4/3\mapsto 5/4...$. 
    These values converge to 1, proving the theorem. 
    \qedhere
    
\end{proof}

\bigskip

\section{Applications to convex sequences}
\label{section-application}

Let us first introduce some notation for this section. For a finite set \(A\) in an additive group, we define the representation function
\[
r_{A+A} (x) = \# \left\{ (a_1,a_2) \in A ^{2} : x = a_1 + a_2  \right\},
\]
and we define \(r_{A-A} \), \(r_{A+A+A} \) similarly. We define the \(k\)-th additive energy of \(A\) as
\[
    E_{k}(A) = \sum _{x}  r_{A-A} (x)^{k}.
\]
A useful identity is
\begin{equation}
\label{eq:sum-E2}
    E_2(A) = \sum _{x} r_{A+A} (x) ^{2}
\end{equation}

\subsection{Auxiliary lemmas}

Note that for a finite convex sequence $A$, there is a strictly convex \(f : \mathbb{R} \to \mathbb{R} \) for which \(f(i) = a_{i} \).  

\begin{corollary}\label{cor:convex-sequence-j3}
    Let \(f: \mathbb{R} \to \mathbb{R} \) be strictly convex, and let \(X \subset  \mathbb{R} \) be finite.
    For any \(\e > 0\), we have
    \[
    J_3(f(X)) \ll _{\e}  \left\lvert X + X + X \right\rvert \left\lvert X \right\rvert ^{3 + \e}
    ,\]
    with the implicit constant independent of \(f\). 
    In particular, for every \(\e > 0\), there exists $C_\e>0$ such that for every finite convex sequence \(A\), we have
    \[
    J_3(A) \leq C_{\e} \left\lvert A \right\rvert ^{4 + \e}.
    \]
\end{corollary}

\begin{proof}
    Call \(\gamma(x) = (x,f(x))\) and \(\Lambda = \gamma(X)\), and let us use the shorthand \(r = r_{\Lambda+\Lambda+\Lambda} \), and \(3X = X+X+X\). We note that
    \[
    J_3(\Lambda) = \sum _{p,q} r ((p,q))^{2}, \qquad J_3(f(X)) \leq \sum _{q} \Big( \sum _{p} r((p,q)) \Big) ^{2}
    .\]
    Note that \(r((p,q)) = 0\) if \(p \not  \in 3X\). 
    Expanding \(J_3(f(X))\) and applying Cauchy-Schwarz, 
    \begin{align*}
        J_3(f(X)) \leq& \sum _{p_1,p_2,q}  r((p_1,q))r((p_2,q)) 1_{3X} (p_1) 1_{3X} (p_2) \\
        \leq&\, \Big( \sum _{p_1,p_2,q} r((p_1,q))^{2} 1_{3X} (p_2) \Big) ^{\frac{1}{2} } \Big( \sum _{p_1,p_2,q} r((p_2,q))^{2} 1_{3X} (p_1) \Big) ^{\frac{1}{2} }\\
        = & \,
        \left\lvert X+X+X \right\rvert J_3(\Lambda).
    \end{align*} 
    The first part then follows by Theorem \ref{thm:main}, and the second part follows by letting \(f\) be such that \(f(i) = a_{i} \), and then applying the first part with \(X = \left\{ 1,2, \ldots , N \right\} \).
    \qedhere
    
\end{proof}

\medskip 

This result on \(J_3\) fits nicely into the existing convex sumset and difference set framework. 
We first record the following lemma, a special case of \cite[Lemma 4]{Rudnev-Stevens}. 
The auxiliary set \(B\) is introduced to account for the fact that both projections in \eqref{eq:projection-maps} produce a difference-set term. 
Without this lemma, the argument only gives
\[
|A+A|^{1/3}|A-A|^{2/3}
\gg_{\e}
|A|^{5/3-\e}.
\]
To obtain a bound involving only the sumset, we combine \eqref{eq:sum-E2} with the lemma below. 
This step is precisely what weakens the exponent from \(5/3\) 
to \(8/5\).

\begin{lemma}
    \label{lem:rudnev-stevens-input}
    Let \(A \subset \mathbb{R} \) be a sufficiently large finite set. 
    For any \(X \subset A\), define
    \[
    \sigma_X = |X| / (8 |X+X| \log|A|) , \qquad P(X) = \left\{ y  : r_{X+X} (y) \geq \sigma _X \cdot |X| \right\}, 
    \]
    \[
    R(X) = \{x \in X : \left\lvert (x + X)\cap P(X) \right\rvert \geq 3 |X| /4  \}
    .\]
    Then there exists \(B \subset A \) of size \(\left\lvert B \right\rvert \geq \left\lvert A \right\rvert  / 2\) for which
    \[
    E_2(B) \leq E_2(R(B)) \log \left\lvert A \right\rvert 
    .\]
\end{lemma}

With this lemma, we recreate the projection lower bounds of \cite{Cushman-sumproduct}.

\begin{lemma}\label{lem:projection-estimate}
    Let \(A \subset \mathbb{R} \) be a sufficiently large finite set. Then
    \[
    \left\lvert A \right\rvert ^{12} \ll \left\lvert A-A \right\rvert ^{3}  \cdot E_3(A) \cdot  J_3(A)
    .\]
    Moreover, there exists \(B \subset A\) with \(\left\lvert B \right\rvert \geq \left\lvert A \right\rvert / 2\) such that for any \(\e > 0\)
    \[
    E_2(B) ^{3} \left\lvert A \right\rvert ^{6 - \e} \ll _{\e} \left\lvert A+A \right\rvert ^{2} \cdot E_3(A) ^{2} \cdot J_3(A)
    .\]
\end{lemma}

\begin{proof}
    We first introduce the maps
    \begin{equation}
    \label{eq:projection-maps}
        \pi_{-} (r,a,a') = (r-a, r-a') , \qquad \pi_{+} (r,r',b) = (r+b , r'+b).
    \end{equation}
    If \(Z \subset A ^{3} \), then Cauchy-Schwarz gives
    \begin{equation} 
    \label{eq:cauchy-schwarz-projection}
        \left\lvert Z \right\rvert ^{2} \leq E_3(A) \left\lvert \pi_{\pm } (Z) \right\rvert .
    \end{equation}
    Indeed, by Cauchy-Schwarz
    \[
    \left\lvert Z \right\rvert \leq  \left\lvert \pi_{\pm } (Z) \right\rvert ^{\frac{1}{2} } \# \left\{ (z_1,z_2) \in Z ^{2} : \pi_{\pm } (z_1) = \pi _{\pm } (z_2) \right\} ^{\frac{1}{2} }
    ,\]
    and for either \(\pi _{\pm } \) the rightmost term is bounded by
    \[
    E_3(A) = \# \left\{ (a_{i} )\in A^{6}: a_1 - a_2 = a_3 - a_4 = a_5 - a_6 \right\} 
    .\]
    The proof of the lemma follows by two applications of this inequality.
    
    \bigskip  
    
    For the difference set estimate, put \(\delta = \left\lvert A \right\rvert / (11 \left\lvert A-A \right\rvert )\). We define
    \[
    P = \left\{ x : r_{A-A} (x) \geq \delta \left\lvert A \right\rvert  \right\} , \qquad R = \left\{ r \in A : \left\lvert \left( r-A \right) \cap P \right\rvert  \geq 2 \left\lvert A \right\rvert / \sqrt{11} \right\} 
    .\]
    By the definition of \(P\), we have \(\sum _{x \in P} r_{A-A} (x) \geq 10 \left\lvert A \right\rvert ^{2} / 11\), and hence partitioning this sum by the first coordinate gives \(\left\lvert R \right\rvert \gg \left\lvert A \right\rvert \). 
    For each \(r \in R\), there are \(\geq 4 \left\lvert A \right\rvert ^{2} / 11\) pairs \((a,a') \in A ^{2}\) with \(r-a,r-a' \in P\). 
    By the definition of \(P\), there are also \(\geq 10 \left\lvert A \right\rvert ^{2} / 11\) pairs \((a,a')\) with \(a'-a \in P\), and hence by inclusion-exclusion and the fact that \(|R| \gg |A|\),
    \[
    Z_{-}  = \left\{ (r,a,a') \in R \times  A^{2} : r-a,r-a',a'-a \in P \right\} 
    \]
    satisfies \(\left\lvert Z_{-}  \right\rvert \gg \left\lvert A \right\rvert ^{3} \). 
    Note that 
    \[
    \left\lvert \pi_{-} (Z_{-} ) \right\rvert \leq \# \left\{ p_1 - p_2 = p_3 : p_{i} \in P \right\} 
    ,\]
    and since all \(p \in P\) have \(\geq \delta \left\lvert A \right\rvert \) representations as a difference, we have
    \[
    \left( \delta \left\lvert A \right\rvert  \right) ^{3}  \left\lvert \pi_{-} (Z_{-} ) \right\rvert \leq  \# \{ (a_1 - a_2) - (a_3 - a_4) = (a_5 - a_6) : a_i \in A  \}= J_3(A)
    .\]
    Applying (\ref{eq:cauchy-schwarz-projection}) and using \(|Z_-| \gg |A|^3\) gives
    \[
    \left\lvert A \right\rvert ^{6} \ll \left\lvert Z_{-}  \right\rvert ^{2} \leq E_3(A) \left\lvert \pi_{-} (Z_{-} ) \right\rvert \ll \left( \delta \left\lvert A \right\rvert \right) ^{-3 } E_3(A) J_3(A).
    \]
    Recalling \(\de = |A| / (11 |A-A| )\), we prove the first estimate.
    
    \bigskip
    For the sumset, we begin by choosing \(B \subset A\) as in Lemma \ref{lem:rudnev-stevens-input}. Dyadic pigeonholing for \(E_2(R(B))\) gives \(\Delta \geq 1\) such that, letting 
    \[
    P_{\Delta} = \left\{ x : r_{R(B) - R(B)} (x) \in [\Delta, 2 \Delta) \right\},
    \]
    we have
    \begin{equation} 
    \label{eq:choice-of-dyadic}
        E_2(B) \leq \log \left\lvert A \right\rvert E_2(R(B)) \ll \left( \log \left\lvert A \right\rvert  \right) ^{2} \Delta ^{2} \left\lvert P_{\Delta} \right\rvert  .
    \end{equation}
    Note that there are \(\gg \De |P_{\De}|\) pairs \((r,r') \in R(B) ^2\) with \(r'-r \in P_{\De}\). 
    For each such pair, by the definition of \(R(B)\) and inclusion-exclusion, there are \(\geq |B| /2\) many \(b \in B\) with \(r + b, r' + b \in P(B)\). 
    Therefore, \(\left\lvert Z_{+}  \right\rvert \gg \Delta \left\lvert P_{\Delta}  \right\rvert  \left\lvert B \right\rvert \), where
    \[
        Z_{+} = \left\{ (r,r',b)\in R(B)^{2}\times B : r+b,r'+b \in P(B), r'-r \in P_{\Delta}  \right\} 
    .\]
    If \((p_1,p_2) \in \pi_{+} (Z_{+} )\), then \(p_1,p_2 \in P(B)\), \(p_2-p_1 \in P_{\Delta} \). As before, expanding all of these into sums and differences in \(A\) we obtain 
    \[
        \left( \sigma_{B} \left\lvert B \right\rvert  \right) ^{2} \Delta \left\lvert \pi_{+} (Z_{+} )  \right\rvert \ll \# \left\{ (a_1 + a_2) - (a_3 + a_4) = (a_5 - a_6) : a_{i}\in A \right\} = J_3(A)
    .\]
    Applying (\ref{eq:cauchy-schwarz-projection}) and using \(|Z_+| \gg \De |P_{\De}|\, |B|\) gives
    \begin{equation*}
        \Delta ^{2} \left\lvert P_{\Delta} \right\rvert ^{2} \left\lvert B \right\rvert ^{2} \ll \left\lvert Z_{+}  \right\rvert ^{2} \leq E_3(A) \left\lvert \pi_{+} (Z_{+} ) \right\rvert \ll \Delta ^{-1} (\sigma _B |B| ) ^{-2} E_3(A) J_3(A).
    \end{equation*}
    
    We interpolate this with  $\Delta ^{3} \left\lvert P_{\Delta}  \right\rvert \leq E_3(A)$, obtaining 
    \[
    \De ^6 |P_{\De}|^3 |B| ^2 \ll  (\sigma _B |B|)^{-2} E_3(A)^2 J_3(A)
    .\]
    Finally, use (\ref{eq:choice-of-dyadic}) and recall that \(\sigma_B = |B| / (8 |B+B| \log |A|)\) to obtain
    \[
    E_2(B) ^{3} \left\lvert B \right\rvert ^{6} \ll \left\lvert B+B \right\rvert ^{2} \left( \log \left\lvert A \right\rvert  \right) ^{8} E_3(A)^{2} J_3(A)
    .\]
    The desired result follows upon using \(\left\lvert B+B \right\rvert \leq \left\lvert A+A \right\rvert \), \(\left\lvert B \right\rvert \gg \left\lvert A \right\rvert \), and suppressing powers of \(\log \left\lvert A \right\rvert \) by \(\left\lvert A \right\rvert ^{\e}\).
\end{proof}

We also use a classical consequence of the Szemer\'edi-Trotter incidence theorem for translations of one convex curve.
See \cite[Corollary 5]{Schoen-Shkredov}
\begin{proposition}\label{prop:sharp-e3}
Let \(A\) be a finite convex sequence. 
Then
\[
    E_3(A) \ll \left\lvert A \right\rvert ^{3}  \log \left\lvert A \right\rvert.
\]
\end{proposition}

\medskip 

\subsection{Proof of Theorem \ref{difference-set-thm}}

Fix \(\e > 0\) and note that by choosing \(c_{\e} \) as a sufficiently small constant, we can account for those \(A\) which are not sufficiently large for Lemma \ref{lem:projection-estimate}. Thus, we may assume that \(A\) is large enough for Lemma \ref{lem:projection-estimate} to hold.

By Lemma \ref{lem:projection-estimate} we have
\[
\left\lvert A \right\rvert ^{12} \ll \left\lvert A-A \right\rvert ^{3}  \cdot E_3(A) \cdot J_3(A)
.\]
Using Corollary \ref{cor:convex-sequence-j3} and Proposition \ref{prop:sharp-e3}, we obtain
\[
\left\lvert A-A \right\rvert \gg _{\e} \left\lvert A \right\rvert ^{5 /3 - \e}
.\]

By Lemma \ref{lem:projection-estimate} again, there is \(B \subset A\) of size \(\left\lvert B \right\rvert \geq \left\lvert A \right\rvert / 2\) for which
\[
E_2(B)^{3} \left\lvert A \right\rvert ^{6 - \e}\ll _{\e}  \left\lvert A+A \right\rvert ^{2} \cdot E_3(A) ^{2} \cdot J_3(A)
.\]
By Cauchy-Schwarz, we have
\[
\left\lvert A \right\rvert ^{2} \ll \left\lvert B \right\rvert ^{2} = \sum _{x} r_{B + B} (x) \leq \left\lvert B+B \right\rvert ^{\frac{1}{2} } E_2(B) ^{\frac{1}{2} } \leq \left\lvert A+A \right\rvert ^{\frac{1}{2} } E_2(B) ^{\frac{1}{2} }
,\]
which combined with \eqref{eq:sum-E2}, Corollary \ref{cor:convex-sequence-j3}, Proposition \ref{prop:sharp-e3} gives the desired
\[
\left\lvert A + A \right\rvert \gg _{\e}  \left\lvert A \right\rvert ^{8 / 5 - \e}. 
\]

\bigskip

\bibliographystyle{alpha}
\bibliography{bibli}

\end{document}